\documentclass[12pt]{amsart}
\usepackage{graphicx}
\usepackage{amsmath}
\usepackage{amssymb}
\usepackage{tikz-cd}

\newtheorem{theorem}{Theorem}[section]
\newtheorem{lemma}[theorem]{Lemma}
\newtheorem{corollary}[theorem]{Corollary}
\newtheorem{proposition}[theorem]{Proposition}

\theoremstyle{definition}
\newtheorem{definition}[theorem]{Definition}

\newtheorem{example}[theorem]{Example}

\numberwithin{equation}{section}
\numberwithin{equation}{section}

\begin{document}

\title[$e$-cohomology modules]{Essential cohomology modules}

\author[R. H. Mustafa]{Runak H. Mustafa}
\address{Mathematics Department, Faculty of Science\\
	Soran University\\ 44008, Soran, Erbil
	Kurdistan Region, Iraq}
\email {rhm310h@maths.soran.edu.iq}

\author[I. Akray]{Ismael Akray}
\address{Mathematics Department, Faculty of Science\\
	Soran University\\ 44008, Soran, Erbil
	Kurdistan Region, Iraq}

\email {akray.ismael@gmail.com}

\begin{abstract}

In this article, we give a generalization to injective modules by using $e$-exact sequences introduced by Akray in [1] and name it $e$-injective modules and investigate their properties. We reprove both Baer criterion and comparison theorem of homology using $e$-injective modules and $e$-injective resolutions. Furthermore, we apply the notion $e$-injective modules into local cohomology to construct a new form of the cohomology modules call it essential cohomology modules (briefly $e$-cohomology modules). We show that the torsion functor $\Gamma_a ( - )$ is an $e$-exact functor on torsion-free modules. We seek about the relationship of $e$-cohomology within the classical cohomology. Finally, we conclude that they are different on the vanishing of their $i_{th}$ cohomology modules.

\end{abstract}

\maketitle

\section{Introduction}
 Exact sequences play an important roles in module theory and homological algebra. Some notions such as injectivity, projectivity, flatness and derived functors have been defined and analyzed by exact sequence approach. Generalizing exact sequences gives possibilities to generalize some related notions which are defined by exact sequences. An exact sequence $0\rightarrow A\stackrel i\to B\stackrel p\to C\to 0$ is split if there exists a morphism $ j:C \rightarrow B $\:(or $f:B \rightarrow A $) such that $ pj=I_C$ (or $\: fi=I_A)$. In 1972, R. S. Mishra introduced a generalization for split sequence where a semi-sequence $M_{i-1}\stackrel{f_{i-1}}\to M_i\stackrel{f_i}\to M_{i+1}$ is called semi-split if $Ker(f_i)$ is a direct summand of $M_i$ \cite{H}. So a semi-split is split if and only if it is exact. In 1999, Davvaz and parnian-Goramaleky introduced a generalization for exact sequences called it a $U$-exact sequence, where a sequence of $R$-modules and $R$-homomorphisms  $\dots \to M_{i-1}\stackrel{f_{i-1}}\to M_i\stackrel{f_i}\to M_{i+1}\to \dots$ is a $U_{i+1}$-exact at $M_i$ if $Im(f_i)=f_{i+1}^{-1}(U_{i+1})$, where $U_{i+1}$ is a submodule of $M_{i+1}$ \cite{DP}. Also, a short sequence $0\to A\stackrel{f} \to B\stackrel{g}\to C\to 0$ of $R$-modules is a $U$-exact if it is ${0}$-exact at $A$, $U$-exact at $B$ and ${0}$-exact at $C$, equivalently, if $f$ is monic, $g$ is epic and $Imf=g^{-1}(U)$ for a submodule $U$ of $C$. Clearly, a $U$-exact sequence is exact if and only if $U=0$. Furthermore, they defined a $V$-coexact sequence as a dual notion of $U$-exact sequence, where a sequence $0\to A\stackrel{f} \to B\stackrel{g}\to C\to 0$ is $V$- coexact if $f$ is monic, $g$ is epic and $f(V)=Ker(g)$, where $V$ is a submodule of $A$. A submodule $N$ of an $R$-module $M$ is called essential or large in $M$ if it has non-zero intersection with every non-zero submodule of $M$ and denoted by $N\leqslant_e M$. Akray in 2020 \cite{AZ} introduced another generalization to exact sequences of modules and instead of the equality of $Im(f)$ with $Ker(g)$ they took $Im(f)$ as a large (essential) submodule of $Ker(g)$ in a sequence $0\to A\stackrel{f} \to B\stackrel{g}\to C\to 0$ and called it essential exact sequence  or simply $e$-exact sequence. Equivalently, a sequence of $R$-modules and $R$-morphisms $\dots\to N_{i-1} \stackrel{f_{i-1}}\to N_i\stackrel{f_i}\to N_{i+1}\to \dots$ is said to be essential exact (e-exact) at $N_i$, if $Im(f_{i-1})\leqslant_e Ker(f_i)$ and to be e-exact if it is e-exact at $N_i$ for all $i$. In particular, a sequence of $R$-modules and $R$-morphisms $0\to L\stackrel{f_1} \to M\stackrel{f_2}\to N\to 0$ is a short e-exact sequence if and only if  $Ker(f_1)=0, Im(f_1)\leqslant_e Ker(f_2)$ and $Im(f_2)\leqslant_e N$. They studied some basic properties of $e$-exact sequences and established their connection with notions in module theory and homological algebra \cite{ZA}. Also, F. Campanini and A. Facchini were worked on $e$-exact sequences and studied the relation of $e$-exactness with some related functors like the functor defined on the category of $R$-modules to the spectral category of $R$-modules and the localization functor with respect to the singular torsion theory \cite{CF}.

In this paper, we continue to observe further properties of $e$-exact sequences and we will restrict our discussion to their applications on both injective modules and the torsion functor of local cohomology.The local cohomology was introduced by Grothendieck in a seminar in Harvard 1961 and written up by Hartshorne in 1967. Next, this subject was studied by Hartshorne and numerous authors even in the recent years see \cite{MR} and \cite{G}.

In section two, we introduce the notion essential injective module briefly $e$-injective module. We prove the Baer Criterion for e-injectives (Theorem 2.2).
We prove that a product of $R$-modules is an $e$-injective if and only if each of its components is an $e$-injective(Proposition \ref{ee}). Also, we show that a torsion-free $R$-module $E$ is an $e$-injective if and only if $ Hom(\quad, E) $ is a contravariant e-exact functor (Proposition \ref{aa}). Moreover,  we reprove the Comparison Theorem of homology for e-injectives (Theorem \ref{bb}). 

Section three devoted for discussing the application of $e$-exact sequence to local cohomology modules. We construct cohomology modules by using $e$-injective resolutions and called the $i_{th}$-local cohomology of an $R$-module $M$ by $_eH^i_a(M)$. We prove that $_eH^0_a(\:)$ is naturally equivalent to $r\Gamma_a(\:)$ on torsion-free $R$-modules (Theorem \ref{3.10}). We study some cases of vanishinig of $e$-cohomology modules (Theorem \ref{ne} and Theorem \ref{ein}). Furthermore, we show that by an example that not necessary an ${a}$-torsion $R$-module $M$ has zero $e$-cohomology modules $_e H^i_a(M)=0$, for $i>0$ as this is the case in local cohomology (Example \ref{ex}).

\section{Essential injective modules}
\hskip 0.6cm

In this section we introduce essential injective module and investigate some properties and results on such topic. We begin with their definition.
\begin{definition}
	We call an $R$-module $E$ essential injective briefly e-injective if it satisfies the following condition: for any monic $f_1:A_1\rightarrow A_2$
	and any map $f_2:A_1\rightarrow E$ of $R$-modules, there exist $0\ne r \in R$ and $f_3:A_2\rightarrow E$ such that $f_3f_1=rf_2$.
	\[\begin{tikzcd}[sep=2cm]
		{} & E & {} \\
		0 \ar{r} & A_1 \ar{u}{f_2} \ar[r,swap]{r}{f_1} & A_2 \ar[dashed,swap]{lu}{f_3}
	\end{tikzcd}\]
In this case, we say the map $f_3$ is essentially extends to the map $f_2$. 	
\end{definition}

Now, we give one of our main results in this section which is the Baer criterion for $e$-injectives.
\begin{theorem}\label{B}[e-Baer Criterion]
	 An $ R- $module $ E $ is e-injective if and only if for a nonzero ideal $\textbf{a}$ of $R$, every $ R $-map $ f:\textbf{a} \to E $ can be extends essentially to $ g:R\to E $ in the sence that there exists $ 0 \ne r \in R $ with $ gi=rf $, where $i:\textbf{a} \to R$ is an inclusion map as in the following diagram.
	\[\begin{tikzcd}[sep=2cm]
		{} & E & {} \\
		0 \ar{r} & \textbf{a} \ar{u}{f} \ar[r,swap]{r}{i} & R\ar[dashed,swap]{lu}{g}		
	\end{tikzcd}\]		
\end{theorem}
\begin{proof}
	Since any ideal of $R$ can be considered as a submodule of the $R$-module $R$, the existence of an extension $g$ of $f$ is just a special case of the definition of e-injectivity of $E$.
	Suppose we have the following diagram, where $A$ is a submodule of $B$ and $i$ is the inclusion map: 
	\[\begin{tikzcd}[sep=2cm]
		{} & E & {} \\
		0 \ar{r} & A \ar{u}{f} \ar[r,swap]{r}{i} & B\ar[dashed,swap]{lu}{g}		
	\end{tikzcd}\]	Let $X$ be the set of all ordered pairs $(A',g')$, where $A\subseteq A'\subseteq B$  and $g':A'\to E$ essentially extend $f$: that is, $g'|_A =rf$ for nonzero r in $R$. Partially ordered $X$ by defining $(A',g')\preceq (A'',g'')$, which means $A'\subseteq A''$ and $g''$ essentially extends $g'$, so the chains in $X$ have upper bounds in $X$, hence using Zorn Lemma, there exists a maximal element $(A_0,g_0)$ in $X$. If $A_0=B$, we are done, if it is not, there is $b\in B$ with $b\notin A_0$ . Define $\textbf{a}=\{r\in R:rb\in A_0\} $, it is clear that $\textbf{a}$ is an ideal in $R$. Define $h:\textbf{a} \to E$, by $h(r)=g_0(rb)$. By hypothesis, there is a map $h^*:R\to E$ essentially extends $h$. Finally, define $A_1=A_0+<b>$ and $g_1:A_1\to E$ by $g_1(a_0+rb)=g_0(a_0)+rh^*(1)$, where $a_0\in A_0$ and $r\in R$. Let us show that $g_1$ is well-defined. If $a_0+rb=a_0'+r'b$, then $(r-r')\in \textbf{a}$. Therefore, $g_0((r-r')b)$ and $h(r-r')$ are defined, and we have $g_0(a_0'-a_0)=g_0((r-r')b)=h(r-r')=h^*(r-r')=(r-r')h^*(1)$. Thus $g_0(a_0')-g_0(a_0)=rh^*(1)-r'h^*(1)$ and $g_(a_0')+r'h^*(1)=g_0(a_0)+rh^*(1)$ as desired. Clearly, $g_1(a_0)=g_0(a_0)$ for all $a_0\in A_0$, so that the map $g_1$ essentially extends $g_0$. We conculde that $(A_0,g_0)\prec (A_1,g_1)$ contradicting the maximality of $(A_0,g_0)$. Therefore, $A_0=B$, the map $g_0$ is an essentially extends $f$ and $E$ is an e-injective.
\end{proof}

The following example shows that an $e$-injective module may not be injective.
\begin{example}
	The $\mathbb{Z}$-module $\mathbb{Z}$ is $e$-injective module, but it is not injective. We can show that $\mathbb{Z}$ is an $e$-injective by using $e$-Baer Criterion. Let $f_1:n\mathbb{Z}\to \mathbb{Z}$ be an inclusion map defined as $f_1(x)=sx; s\in \mathbb{Z}$ and $f_2:n\mathbb{Z}\to\mathbb{Z}$ defined as $f_2(x)=mx; m\in \mathbb{Z}$ where $(m,s)=1$. Then by talking $f_3=f_2$, we have $f_3\circ f_1(x)=f_3(f_1(x))=f_3(sx)=msx=sf_2(x)$
	\[\begin{tikzcd}[sep=2cm]
		{} & \mathbb{Z} & {} \\
		0 \ar{r} & n\mathbb{Z} \ar{u}{f_2} \ar[r,swap]{r}{f_1} & \mathbb{Z} \ar[dashed,swap]{lu}{f_3}
	\end{tikzcd}\]
\end{example}
It's esay to see that every submodule of $\mathbb{Z}$-module $\mathbb{Z}$ is $e$-injective. On an integral domain, every injective is divisble but this is not the case for e-injectives, since $\mathbb{Z}$ as $\mathbb{Z}$-module is $e$-injective while not divisble.

\begin{definition} {\cite{AZ}}
	An $e$-exact sequence $0\rightarrow A\stackrel i\to B\stackrel p\to C\to 0$ is $e$-split if there exist $0\ne s \in R$ and a morphism $ j:C \rightarrow B $\:(or $f:B \rightarrow A $) such that $ pj=sI_C$ (or $\: fi=sI_A)$.
\end{definition}
\begin{proposition}\label{dd}
	If an e-exact sequence {$ 0\rightarrow A \stackrel{i}\rightarrow B \stackrel{p}\rightarrow C\rightarrow 0 $} is $e$-split, then there exists $0\ne r\in R$ such that $rB\cong A\bigoplus C$.
	\end{proposition}
\begin{proof}
	We show that $rB=Im i \bigoplus Im j$ there exists $0\ne r\in R$ and $j:C \to B$ with $pj=rI_C$. For any $b\in B$, $rb-jp(b)\in Ker p$ and by $e$-exactness, there exists $a\in A$ and $0\ne s\in R$ such that $i(a)=s(rb-jp(b))$, that is $srb=i(a)+sjp(b)$. Hence $srB=Im i+Im j$. Now, if $i(x)=rz=j(y)$ for $x\in A$ and $y\in C$, then $0=pi(x)=p(rz)=pj(y)=ry$ and so $rz=j(y)=0$ which implies that $Im i \cap Im j=0$. Therefore $rB= Im i\bigoplus Im j \cong A \bigoplus C$. 
\end{proof}
In the following proposition we generalize \cite[Proposition 3.38]{R} to $e$-injective $R$-modules.
\begin{proposition}\label{ee}
	A direct product of $R$-modules is an e-injective if and only if each $ E_i$ is an e-injective $R$-module.
\end{proposition}
\begin{proof}
	Suppose that $E=\prod E_i$ , $k_i:E_i\to E$ is the injection and $p_i:E\to E_i$ is the projection such that $p_ik_i=I_{E_i}$. Since $E$ is an e-injective, there exist a homomorphism $h:C\to E$ and $0\ne r\in R$ such that $hj=r(k_i\circ f_i)$. Now, $g\circ j=(p_i\circ h)\circ j=p_i\circ (h\circ j)=p_i(r_i\circ f_i)) =r((p_i\circ k_i)\circ f_i)=rf_i$.
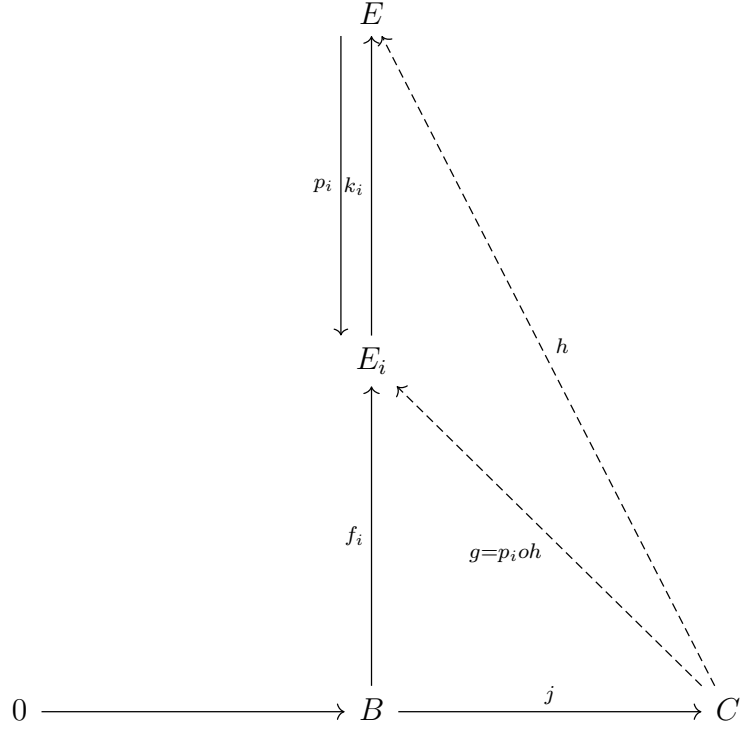
	\[\begin{tikzcd}[sep=4cm]
	{} & E   & {} \\
	{} & E_i \ar{u}{k_i}\ar[u,<-,shift left=4,"p_i"] & {} \\
	0 \ar{r} & B \ar{u}{f_i} \ar{r}{j} & C \ar[dashed, swap]{lu}[swap]{g=p_ioh}\ar[dashed,swap]{luu}{h}
\end{tikzcd}\]	
Conversely, consider we have the following diagram in which $ E=\prod E_k $ 
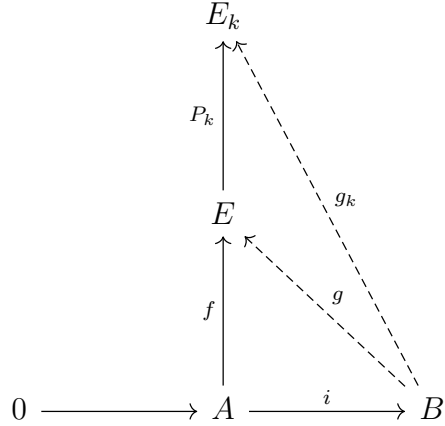
	\[\begin{tikzcd}[sep=2cm]
	{} & E_k  & {} \\
	{} & E \ar{u}{P_k} & {} \\
	0 \ar{r} & A \ar{u}{f} \ar{r}{i} & B \ar[dashed, swap]{lu}{g}\ar[dashed,swap]{luu}{g_k}
\end{tikzcd}\]	
Let $ p_k:E\rightarrow E_k $ be the $k_{th}$ projection, so that $ p_k f: A \rightarrow E_k$, since $ E_k $ is an e-injective then there exist $0\ne r\in R$ and $ g_k:B\rightarrow E_k$ such that $g_ki=r(p_kf)$. Now, we define $g:B\to E$ by $ g:b\mapsto <g_k (b)> $ the map $ g $ does essentially extends $ f $, since $ g(b) =g(i(a)) =<g_k i(a)>= r<p_k f(a)>= r f(a) $.
\end{proof}
Recall from \cite{AZ} the triangle diagram of $R$-modules an $R$-morphism as the following are $e$-commute if and only if there exists $0\ne r\in R$ such that $g\circ i=rf$.
\[\begin{tikzcd}[sep=2cm]
	{} & B & {} \\
	0 \ar{r} & A_1 \ar{u}{f} \ar[r,swap]{r}{i} & A_2 \ar[swap]{lu}{g}
\end{tikzcd}\]
As well too, a squar diagram as the following are $e$-commute if and only if there exists $0\ne r \in R$ such that $qf=rgt$
\[ \begin{tikzcd}[arrows={-Stealth}]
	A_1\rar["f"]\dar["t"] & A_2\dar["q"]\\
	B_1\rar["g"] & B_2
\end{tikzcd}	\]

\begin{proposition}\label{aa}
	Let $E$ be a torsion-free $R$-module. Then $ E $ is an e-injective if and only if $ Hom(\quad, E) $ is a contravariant e-exact functor.
\end{proposition}
\begin{proof}
	Suppose that {$ 0\rightarrow A \stackrel{i}\rightarrow B \stackrel{p}\rightarrow C\rightarrow 0 $} is a short e-exact sequence. Since $E$ is a torsion-free, $Hom(\quad,E)$ is a left e-exact functor by \cite[Theorem 2.7]{AZ}. It remains to show that {$ Hom(\quad,E) $} is a right e-exact, which means {$ Hom(B,E)\stackrel{i^{*}}\rightarrow Hom(A,E)\rightarrow 0 $} is an e-exact. For this purpose, let {$ 0 \ne f \in Hom(A,E) $}. $e$-injectivity of $E$ implies that there exist $ g:B \rightarrow E $ and $ 0 \ne r \in R $ such that $ i^(*g)=gi=rf $. Thus we have $ Im i^* \leqslant_e Hom(A,E) $. Therefore $ Hom(\quad,E)$ is an e-exact functor. Conversely, if $ Hom(\quad,E)$ is an e-exact contravariant functor so the sequence $ 0 \rightarrow Hom(C,E)\stackrel{p^{*}}\rightarrow Hom(B,E) \stackrel{i^{*}}\rightarrow Hom(A,E) \rightarrow 0$ is an e-exact sequence. Then for all $ 0 \ne r\in R$ and $ 0 \ne f \in Hom(A,E) $ there exists $ g \in Hom(B,E) $ such that $ i^*g=rf $ which implies that $ gi=rf $, that is, the diagram below
			\[\begin{tikzcd}[sep=2cm]
		{} & E & {} \\
		0 \ar{r} & A \ar{u}{f} \ar[r,swap]{r}{i} & B \ar[dashed,swap]{lu}{g}
	\end{tikzcd}\] is an e-commute and $ E $ is an e-injective.
	\end{proof}

\begin{proposition}\label{ff}
	Let $E$ be a torsion-free $e$-injective $R$-module. Then any e-exact sequence $0\to E\to B\to C\to 0$ $e$-splits.	
\end{proposition}
\begin{proof}
 Let $E$ be an $e$-injective $R$-module and the sequence $0\rightarrow E\stackrel{i}\to B \stackrel{p}\to C\to 0$ be an e-exact. Then by Proposition \ref{aa}  $0\rightarrow Hom(C,E)\stackrel{p^{*}}\rightarrow Hom(B,E) \stackrel{i^{*}}\rightarrow Hom(E,E)\rightarrow 0$ is an e-exact sequence. Since $I_E\in Hom(E,E)$, there exist $g\in Hom(B,E)$ and $0\ne r\in R$ such that $i^*(g)=gi=rI_E$ and so the given sequence $e$-splits.
\end{proof}

\begin{theorem}
	Given an e-commutative diagram of $R$-modules having $e$-exact rows and torsion-free module $C''$:
	\[ \begin{tikzcd}[arrows={-Stealth}]
		A'\rar["i"]\dar["f"] & A\rar["p"]\dar["g"] & A''\dar[dashed]{h}\rar & 0\\%
		C'\rar["j"] & C\rar["q"] &	C''\rar &0
	\end{tikzcd}	\] there exist a unique map $h:A''\to C''$ making the augmented diagram $e$-commute.   
\end{theorem}
\begin{proof}
	If $a''\in A''$, then there exist $a\in A$ and $0\ne r\in R$ such that $p(a)=ra''$. Define $h(a'')=rqg(a)$. We must shows that $h$ is well defined, that is, if $u\in A$ satiesfies $p(u)=ra''$. Now, $p(a)=p(u)$ implies $p(a-u)=0$, so $a-u\in Ker p$ by $e$-exactness, there exist $0\ne S\in R$ and $a'\in A$ such that $i(a')s(a-u)$. Thus, $rsqg(a-u)=rqg(i(a'))=qjf=0$, because $qj=0$. Therefore, $h$ is well-defined. Suppose that $h':A''\to C''$ be another map satisfies $rh'p=qg$ and $a\in A$. Then $rh'p(a)=qg=rhp(a)$ and so $h$ is unique.
	
\end{proof}
\begin{theorem}
	Given an e-commutative diagram of $R$-modules having $e$-exact rows 
	\[ \begin{tikzcd}[arrows={-Stealth}]
		0\rar & A'\rar["i"]\dar[dashed]{f} & A\rar["p"]\dar["g"] & A''\dar["h"]\\%
		0\rar & C'\rar ["j"] & C\rar["q"] &	C''
	\end{tikzcd}	\] there exist a unique map $f:A'\to C'$ making the augmented diagram $e$-commute.Moreover, if $g$ and $h$ are isomorphism and $A$ and $A''$ are torsion-free, then $A'\cong rC'$ for some non-zero elements $0\ne r\in R$.
\end{theorem}
\begin{proof}
	Let $a', u'\in A'$, define $f(a')=rqg(i(a'))$. We must shows that $f$ is well defined, let $i(a')=i(u')$ implies $qg(i(a'-u'))=0$, so $g(i(a-u))\in Ker\: q$ by $e$-exactness, there exist $0\ne r\in R$ and $c'\in C'$ such that $j(c')=rg(i(a'-u'))$. Thus, $rqg(i(a'-u'))=qj(c')=0$, because $qj=0$. Therefore, $f$ is well-defined. Suppose that $f':A'\to C'$ be another map satisfies $jf'=rgi$. Then $jf'(a')=rgi(a')=jf(a')$ this implies that $(f'(a')-f(a'))\in Kerj=0$ and so $f$ is unique. \\
	Let $a'\in Ker\:f$. Then $f(a')=0$ and by $e$-commutativity there exists $0\ne r\in R$ such that $jf(a')=0=rgi(a')$. Thus $r.i(a')\in Ker\:g=0$ and $A$ is torsion-free so $i(a')=0$ implies $a'\in Ker\:i=0$. Suppose that $c'$ is a non-zero element of $C'$. Since $j(c')\in C$ and $g$ is onto, then there exist $a\in A$ such that $g(a)=j(c')$, $e$-commutativity gives us $qg(a)=rhp(a)$, for some $0\ne r\in R$. Now, $qg(a)=qj(c')=0=rhp(a)$, so $rp(a)\in Ker\:h=0$ torsion-freeness of $A''$ gives us $a\in Ker\:p$, then there exist $0\ne s\in R$ and $a'\in A'$ such that $i(a')=sa$. By $e$-commutativity there exist $0\ne t\in R$ such that $jf(a')=tgi(a')$ so $jf(a')=tg(i(a'))=tsg(a)=tsj(c')$. We obtain $(f(a')-tsc')\in Ker\:j=0$ and so $f(a')=tsc'$. Therefore, $A'\cong rC'$ for some non-zero element $0\ne r\in R$.  
\end{proof}
\begin{lemma}
	Consider the commutative diagram of $R$-modules and $R$-morphisms, where $R$ is a domain, $B$ and $B''$ are torsion-free $R$-modules
	\[ \begin{tikzcd}[arrows={-Stealth}]
		{}& 0\dar&0\dar & 0\dar\\
		0\rar & A\rar["f"]\dar["i"] & A'\rar["f'"]\dar["j"]& A''\rar \dar["p"] & 0\\
		0\rar & B\rar["g"]\dar["i'"]& B'\rar["g'"]\dar["j'"] & B''\rar \dar["p'"]& 0\\
		0\rar & C\rar["h"]\dar& C'\rar["h'"]\dar & C''\rar \dar & 0\\
		{}&0&0 & 0&{}
		\\	\end{tikzcd}	\] If the columns and the first and the third rows are $e$-exact, then the middle row is also $e$-exact.
\end{lemma}
\begin{proof}
	To proof the middle row is $e$-exact, we have to check the following three conditions:\\
	1) $Ker(g)=0$. Take $b\in Ker(g)$. Then $j'g(b)=0=hi'(b)$ and $i'(b)\in Ker(h)$. Since $Ker(h)=0$, so $i'(b)=0$ and $b\in Ker(i')$ and as $ Im(i)\leqslant_e Ker(i')$, then there exist a non-zero element $r$ belongs to $R$ and $a\in A$ such that $i(a)=rb$ and so $gi(a)=g(rb)=rg(b)=0=jf(a)$, so $f(a)\in Ker(j)$. Since $Ker(j)=0$, $f(a)=0$ and $a\in Ker(f)=0$, which means $a=0$ this implies $rb=0$ and $b=0$, since $B$ is torsion-free. Therefore, $g$ is monic.
	
	2) $Im(g)\leqslant_e Ker(g')$. First to prove $Im(g)\subseteq Ker(g')$. Let $b'\in Im(g)$. Then there exists $b\in B$ such that $g(b)=b'$ and $j'g(b)=hi'(b)=j'(b')$, which means that $j'(b')\in Im(h)\subseteq Ker(h')$ and so $p'g'(b')=h'j'(b')=0$. Hence $g'(b')\in Ker(p')$ and as $Im(p)\leqslant_e Ker(p')$, then there exist $a''\in A''$ and a non-zero element $s$ belongs to $R$ such that $p(a'')=sg'(b')=s(g'g(b))=0$ and $sg'(b')=0$, so $g'(b')=0$, since $B''$ is torsion-free. Thus $b'\in Ker(g')$ and so $Im(g)\subseteq Ker(g')$. Now, for essentiality, take $b'$ to be a non-zero element of $Ker(g')$. Then $0=g'(b')=g'j(a')=pf'(a')$ and $f'(a')\in Ker(p)$. Since $Ker(p)=0, f'(a')=0$ so $a'\in Ker(f')$ and as $Im(f)\leqslant_e Ker(f')$ then there exist a non-zero element $r$ belongs to $R$ and $a\in A$ such that $f(a)=ra'$. Now, $jf(a)=j(ra')=rj(a')=rb'=gi(a)=g(b)$. Therefore, $Im(g)\leqslant_e Ker(g')$.
	
	3) $Im(g')\leqslant_e B''$. Let $b''$ be a non-zero element of $B''$ and $p'(b'')\in C''$. Then there exist $c'\in C'$ and a non-zero element $r$ belongs to $R$ such that $h'(c')=rp'(b'')$ and as $Im(j')\leqslant_eC'$, then there exist a non-zero element $s$ belongs to $R$ and $b'\in B'$ such that $j'(b')=sc'$. Now, we have $p'g'(b')=h'j'(b')=h'(sc')=sh'(c')=srp'(b'')$ and $p'(g'(b')-srb'')=0$ which means $(g'(b')-srb'')\in Ker(p')$ and as $Im(p)\leqslant_e Ker(p')$ then there exist a non-zero element $k\in R$ and $a''\in A''$ such that $p(a'')=k(g'(b')-srb'')$. Also, we have $f'(a')=ta''$, for a non-zero element $t$ belongs to $R$ and $a''\in A''$, because $Im(f')\leqslant_e A''$. Thus $pf'(a')=p(ta'')=tp(a'')=k(g'(b')-srb'')=g'j(a')=g'(b')$ which means $g'(kb'-b')=ksrb''$. Therefore, $Im(g')\leqslant_e B''$.  
\end{proof}
\begin{lemma}
	Consider the commutative diagram of $R$-modules and $R$-morphisms, where $R$ is a domain, $C$, $C'$ and $C''$ are torsion-free $R$-modules
	\[ \begin{tikzcd}[arrows={-Stealth}]
		{}& 0\dar&0\dar & 0\dar\\
		0\rar & A\rar["f"]\dar["i"] & A'\rar["f'"]\dar["j"]& A''\rar \dar["p"] & 0\\
		0\rar & B\rar["g"]\dar["i'"]& B'\rar["g'"]\dar["j'"] & B''\rar \dar["p'"]& 0\\
		0\rar & C\rar["h"]\dar& C'\rar["h'"]\dar & C''\rar \dar & 0\\
		{}&0&0 & 0&{}
		\\	\end{tikzcd}	\] If the columns and the above rows are $e$-exact, then the last one is also $e$-exact.
\end{lemma}
\begin{proof}
	To proof the third row is $e$-exact, we have to check the following three conditions:\\
	1) $Ker(h)=0$. Take $c\in Ker(h)$, then $h(c)=0$ and $hi'(b)=j'g(b)$ and and as $Im(i')\leqslant_e C$ then there exist $b\in B$ and a non-zero element $r$ belongs to $R$ such that $i'(b)=rc$. Then $hi'(b)=h(rc)=rh(c)=0=j'g(b)$ so $g(b)\in Ker(j')$ and as $ Im(j)\leqslant_e Ker(j')$, then there exist a non-zero element $s$ belongs to $R$ and $a'\in A'$ such that $j(a')=sg(b)$ and so $g'j(a')=sg'g(b)=0=pf'(a')$, so $f'(a')\in Ker(p)$. Since $Ker(p)=0$, $f'(a')=0$ and $a'\in Ker(f')$ and also by essentiality there exist $a\in A$ and a non-zero elemeny $t$ belongs to $R$ such that $f(a)=ta'$. Then $jf(a)=j(ta')=tj(a')=tsg(b)=gi(a)$ and so $g(tsb-i(a))=0$ which means $tsb-i(a)\in Ker(g)=0$ this implies $tsb=i(a)$ and so $tsi'(b)=i'i(a)=0$, then $i'(b)=0$ since $C$ is torsion-free. Thus $rc=0$ and so $c=0$. Therefore, $h$ is monic.
	
	2) $Im(h)\leqslant_e Ker(h')$. First to prove $Im(h)\subseteq Ker(h')$. Let $c'\in Im(h)$. Then there exists $c\in C$ such that $h(c)=c'$ and from essentiality there exist $b'\in B'$ and a non-zero element $r$ belongs to $R$ such that $j'(b')=rc'$ and so  $h'j'(b')=h'(rc)=rh'(c')=rh'(h(c))=p'g'(b')$ and also by commutativity and essentiality $g'j(a')=g'(tb')=pf'(a')$ for a non-zero $t$ belongs to $R$ and $b'\in Ker(j')$, which means $Im(g')\subseteq Im(p)\subseteq Ker(p')$ and so $p'g'(tb')=0$. Therefore $h'(rc')=rh'(h(c))=tp'g'(b')=0$ and so $rh'(h(c))=0$ so $h'(h(c))=0$, since $C''$ is torsion-free. Hence $c'=h(c)\in Ker(h')$. Now, for essentiality, take $c'$ to be a non-zero element of $Ker(h')$. Then $h'j'(b)=p'g'(b)$ and as $Im(j')\leqslant_e C'$ there exist a non-zero element $r$ belongs to $R$ and $b'\in B'$ such that $j'(b')=rc'$. Hence $hi'(b)=j'g(b)$, since $Im(i')\leqslant_e C$ and $Im(g)\leqslant_e Ker(g')\in B'$ so there are  non-zero elements $c\in C, b\in B$ and a non-zero elements $s$ and $k$ belongs to $R$ such that  $hi'(b)=h(sc)=sh(c)=j'(kb')=krc'$. Therefore, $sh(c)=kr(c')$ put $k=s$ so $s(h(c)-rc')=0$ and so $h(c)=rc'$, since $C'$ is torsion-free. Then $Im(h)\leqslant_e Ker(h')$.
	
	3) $Im(h')\leqslant_e C''$. Let $c''$ be a non-zero element of $C''$ and as $Im(p')\leqslant_e C''$, there exist $b''\in B''$ and a non-zero element $r$ belong to $R$ such that $p'(b'')=rc''$ and as $Im(g')\leqslant_e B''$, then there exist a non-zero element $s$ belong to $R$ and $b'\in B'$ such that $g'(b')=sb''$ and also we have $Im(j')\leqslant_e C'$, there exist a non-zero element $t$ belong to $R$ and $b'\in B'$ such that $j'(b')=tc'$. Hence $h'j'(b')=h'(tc')=th'(c')=p'g'(b')=p'(sb'')=sp'(b'')=src''$, put $s=t$, so $t(h'(c'-rc''))=0$ this implies $h'(c'-rc'')=0$, since $C''$ is torsion-free. Thus, $h'(c')=rc''$ and $Im(h')\leqslant_e C''$.
\end{proof}
	Recall from \cite{AZ} that an e-injective resolution of an $ R- $module $ A $ is an e-exact sequence $ 0\to A\stackrel \eta \to E^0 \stackrel{d^{0}}\to E^1 \stackrel{d^{1}}\to ...\to E^n \stackrel{d^{n}}\to E^{n+1} \to ... $ where each $E^i$ is an e-injective $R$-module. Let $f,g:X\to E$ be two chain maps. Then $f$ is $e$-homotopic to $g$ if there are maps $s^{n+1}:X^{n+1}\to E^n$ and non-zero elements $s$ and $r$ in $R$ such that $r(h^n-f^n)=s^{n+1}d^n+sd^{n-1}s^n$ for all $n$. Now, we are in a position to prove the new form of comparsion theorem by using $e$-injectivity and $e$-exact sequences.
	
\begin{theorem}\label{bb}[Comparison Theorem for e-injectives]\\
	Suppose that we have the following diagram:\\	
		\[ \begin{tikzcd}[arrows={-Stealth}]
		0\rar & A\rar["\varepsilon"] & E^0\rar["d^0"] & E^1\rar["d^1"] & E^2\rar & ...\\
		0\rar & A'\rar["\varepsilon'"]\uar["f"] & X^0\rar["d'^{0}"]\uar[dashed]{f^0} & X^1\uar[dashed]{f^1}\rar["d'^1"] & X^2\uar[dashed]{f^2}\rar & ...
			\end{tikzcd}	\]where the rows are complexes. If each
	 $E^n$, in the top row is e-injective and the bottom row is e-exact, then there exists a chain map $f:X^{A'}\to E^A$, 	   (the dashed arrows) making the completed diagram e-commute. Furthermore, any two such chain maps are e-homotopic.
\end{theorem}
\begin{proof}
	We prove the existence of $ f^n $ by induction $ n\geqslant 0 $. For the base step $ n=0 $, consider the following diagram:
		\[\begin{tikzcd}[sep=2cm]
		{} & E^0 & {} \\
		0 \ar{r} & A' \ar{u}{\varepsilon\circ f} \ar[r,swap]{r}{\varepsilon'} & X_0 \ar[dashed,swap]{lu}{f_0}
	\end{tikzcd}\]	
	 
	Since $ \varepsilon'$ is monic and $E^0$ is e-injective, there exist $f^0:X^0\to E^0$ and $0 \ne r \in R$ such that $f^0\varepsilon'=r(\varepsilon \circ f)$. For the inductive step, consider we have $f^{n-1}$ and $ f^n$ and the following diagram:
		\[ \begin{tikzcd}[arrows={-Stealth}]
		{} & E^{n-1}\rar["d^{n-1}"] & E^n\rar["d^n"] & E^{n+1}\rar["d^{n+1}"] & ...\\
		X^{n-2}\rar & X^{n-1}\rar["d'^{n-1}"]\uar["f^{n-1}"] & X^n\rar["d'^{n}"]\uar{f^{n}} &X^{n+1}\uar[dashed]{f^{n+1}}
	\end{tikzcd}	\]
	
	Since $ d^nf^n(d{'^{n-1}}X^{n-1})=r(d^nd^{n-1}f^{n-1}(X^{n-1}))=0$, $ Imd^{'n-1} \subseteq Ker d^nf^n$ and as $ Im d^{'n-1}
	\leqslant_e Ker d^n $,we have the following diagram\\
		\[\begin{tikzcd}[sep=2cm]
		{} & E^{n+1} & {} \\
		0\ar{r} & Ker d'^n \ar{u}{d^nf^n}\ar[r,swap]{r}{d'^n} & X^{n+1} \ar[dashed,swap]{lu}{f^{n+1}}
	\end{tikzcd}\]	
	in which $ d'^n$ is monic and $ E^{n+1}$ is e-injective, so there exist $r_n\in R \quad$ and $f^{n+1}:X^{n+1} \to E^{n+1}$ such that $ f^{n+1}d^{'n}=r_nd^nf^n$.  
	Now, to show the uniqueness of $f$ up to e-homotopy. If $h:X^{A'}\to E^{A'}$, is another chain map with $h^0\varepsilon'=r\varepsilon f$, we construct the terms $s^n:X^{n+1}\to E^n$ of an e-homotopy $s=s^n$ by induction on $ n\ge0 $ that we will show that $ s(h^n-f^n)=s^{n+1}d^{n'}+r'd^{n+1}s^n$ for $s,r',p$ and $q$ in  $R$. We define $s^0 : X^0\to A$ to be the zero map. Now to show that $Im (r(h^n-f^n)-r'd^{n-1}s^n)\subseteq Ker d^n$, we have $ d^n(r(h^n-f^n)-r'd^{n-1}s^n\\ =d^n(r(h^n-f^n)-d^n(r'd^{n-1}s^n)$\\
	$=rd^n(h^n-f^n)-r'd^n(d^{n-1}s^n)$\\
    $=rd^n(h^n-f^n)-r'd^n(p(h^n-f^n)-qs^{n+1}d'^n)$\\
      $=rd^n(h^n-f^n)-r'pd^n(h^n-f^n)-r'qd^n(s^{n+1}d'^n)$\\
        $=rd^n(h^n-f^n)-r'pd^n(h^n-f^n)-r'qd^n(md^{n-1}s^n)=0$, put $r'p=r$\\
   Therefore, we have the following diagram:
		\[\begin{tikzcd}[sep=2cm]
		{} & E^{n} & {} \\
		0\ar{r} & Ker d'^n \ar{u}{r(h^n-f^n)-r'd^{n-1}s^n}\ar[r,swap]{r}{d'^n} & X^{n+1} \ar[dashed,swap]{lu}{s^{n+1}}
	\end{tikzcd}\]	
	since $E^n$ is e-injective and $d'^{n}$ monic, there exist $s^{n+1}$, $s\in R$ such that $ s^{n+1}d^{n'}=s(r(h^n-f^n-r'd^{n-1}s^n))$. Therefore, $sr(h^n-f^n)=s^{n+1}d^{n'}+sr'd^{n-1}s^n$. Hence $f$ and $h$ are e-homotopic.
\end{proof}

\begin{proposition}\label{1.8}
	Let $I'^{\bullet}$ and $I''^{\bullet}$ be two e-injective resolutions of $ M'$ and $M''$ respectively. Suppose that $ 0\to M'\to M\to M''\to 0$ is an e-exact sequence. Then there exists an e-injective resolution $ I^\bullet$ such that the following diagram
	
	\[ \begin{tikzcd}[arrows={-Stealth}]
		0\rar & M'\rar["i"]\dar["\delta'"] & M\rar["p"]\dar & M''\dar["\delta''"]\rar & 0\\%
		0\rar& I'^\bullet\rar & I^\bullet\rar &	I''^\bullet\rar &0
	\end{tikzcd}	\]	is e-commutative in which the bottom row is an e-exact sequence of complexes.
	
\end{proposition}
\begin{proof}
	Consider the following diagram in which the rows are e-exact
		\[ \begin{tikzcd}[arrows={-Stealth}]
		0\rar & M'\rar["i"]\dar["id_m"] & M\rar["S''op"]\dar[dashed,"f^0"] & I''^0\dar[dashed,"f^1"]\rar["S''^0"] & I''^1\rar["s''^1"]\dar[dashed,"f^2"] & I''^2\rar\dar[dashed,"f^3"] & ...\\%
		0\rar& M'\rar["S'^{-1}"] & I'^0\rar["S'^0"] & I'^1\rar["S'^1"] &	I'^2\rar["S'^2"] & I'^3\rar & ...
	\end{tikzcd}	\]
	By the comparison theorem for e-injectives (Theorem \ref{bb}), there exist maps (dashed arrows) that make all squares e-commute. Now, we define $ I^n=I'{^n}\bigoplus I''{^n}, \delta^{-1}:M \to I^0 $ by $\delta^{-1}(m)=(-f^0(m),\delta''^{-1}\circ p(m))$ and $\delta^n:I^n\to I^{n+1} $ by $\delta^n(a,b)=(\delta'^n(a)+(-1)^nf^{n+1}(b),\delta''^n(b)) $,
			\[ \begin{tikzcd}[arrows={-Stealth}]
		0\rar & M'\rar["i"]\dar["\delta'^{-1}"] & M\rar["p"]\dar[dashed,"\delta^{-1}"] & M''\dar["\delta''^{-1}"]\rar &0\\%
			0\rar & I'^0\rar["i^0"]\dar["\delta'^0"] & I^0\rar["p^0"]\dar[dashed,"\delta^0"] & I''\dar["\delta''^0"]\rar &0\\%
		0\rar& I'^1\rar["i^1"] & I^1\rar["f^1"] & I''^1\rar &0
	\end{tikzcd}	\]

 since $Ker\delta''^{-1}=0$ and $Ker\delta'^{-1}=0$, $Ker\delta^{-1}=0$. Thus $\delta^{-1}$ is monic and from Proposition \ref{bb}, we get $I^n$ is e-injective for each $n\geqslant 0$.
  To prove that $Im\delta^{-1}\leqslant_e Ker\delta^0$ we must show first that $Im\delta^{-1}\subseteq Ker\delta ^0$. Thus $\delta ^0(\delta^{-1}(m))=\delta^0(-f^0(m),\delta''^{-1}\circ p(m))=\delta'(-f^0(m)+(-1)^0f^1(\delta''^{-1}\circ p(m)), \delta''^0(\delta''^{-1}\circ p(m)))=0$,
   which implies that $Im\delta^{-1}\subseteq Ker\delta^0$. Now, for essentiality, let $0\ne(a,b)\in Ker\delta^0$. Then $0=\delta^0(a,b)=(\delta'^0(a)+(-1)^0f^1(b),\delta''^0(b))$, so $\delta'^0(a)+f^1(b)=0$ and $\delta''^0(b)=0$. Since $Im\delta''^{-1}\leqslant_e Ker\delta''^0$, there exists $0 \ne r\in R$ such that $\delta''^{-1}(m'')=rb$ for some $m''$ in $M''$ and $\delta''^{-1}(p(m))=rb$ for some $m\in M$. So $f^1(\delta''^{-1}(p(m)))=f^1(rb)=rf^1(b)=-r\delta'^0(a)$ and by e-commutity there exists $0\ne q\in R$ such that $-r\delta'^0(a)=q\delta'^0f^0(m)$ and this implies that $r\delta'^0(a)+q\delta'^0f^0(m)=0=\delta'^0(ra+qf^0(m))$. Thus we get $ra+qf^0(m)\in Ker\delta'^0$,
  and since $Im\delta'^{-1}\leqslant_e Ker\delta'^0$, there exists $0\ne t\in R$ such that $\delta'^{-1}(m')=tra+tqf^0(m)$ for some $m'\in M'$. Now, $\delta^{-1}(tqm-i(m'))=(-f^0(tqm-i(m')),\delta''^{-1}\circ p(tqm-i(m'))=(-f^0(tqm)+f^0(i(m'))),\delta''^{-1}p(tqm)-\delta''^{-1}p(i(m')))=tra-\delta'^{-1}(m')+p\delta'^{-1}(m'),tqrb)=s(a,b)$, put $q=p=1$
    which proves that $Im\delta^{-1}\leqslant_e Ker\delta^0$ for $n\geqslant 0$ and $\delta^{n+1}(\delta^n(a,b))=\delta^{n+1}(\delta'^{n}(a)+(-1)^nf^{n+1}(b)),\delta''^n(b))\\=\delta'^{n+1}(\delta'^n(a)+(-1)^nf^{n+1}(b)+(-1)^{n+1}f^{n+2}(\delta''^n(b)),\delta''^{n+1}(\delta''^n(b)))=0$.
    This proves that $Im\delta^n\subseteq Ker\delta^{n+1}$. Now let $(a,b)\in Ker\delta^{n+1}$. Then $\delta''^{n+1}(b)=0$ and $\delta'^{n+1}(a)+(-1)^{n+1}f^{n+2}(b)=0$. It follows that $\delta''^n(c)=rb$, $0\ne r\in R$ for some $c\in I''^{n}$. Then 
    $(-1)^ny\delta'^{n+1}f^{n+1}(c)=(-1)^nf^{n+2}(\delta''^n(c))=r\delta'^{n+1}(a)$ 
    so that $\delta'^{n+1}(ra-(-1)^{n+1}yf^{n+1}(c))=0$ and $ra-(-1)^n yf^{n+1}(c)\in Ker\delta'^{n+1}$. Since $Im\delta'^n\leqslant_e Ker\delta'^{n+1}$,
     there exists $0\ne q\in R$ such that $\delta'^n(d)=q(ra-(-1)^n yf^{n+1}(c))$ for some $d\in I'^n$. Thus $\delta^n(d,c)=(\delta'^n(d)+(-1)^nf^{n+1}(c),\delta''^n(c))=r(a,b)$, put $q=y=1$, which proves that $Im\delta^n\leqslant_e Ker\delta^{n+1}$.  
\end{proof}


\section{The cohomology regarding to $e$-exact sequences}
In this section all rings are Noetherian domain and all modules are unitary $R$-modules. We want to describe the right derived functors to the additive covariant functor (torsion functor) $\Gamma_a$ of local cohomology with e-injective resolutions and then we present new forms of some theorems of cohomology with $e$-exact sequences. Let $a$ be an ideal of $R$. The following is a useful result that shows the functor $\Gamma_a(\:)$ is an $e$-exact under suitable condition. 
\begin{lemma}\label{1.13}
	Let $0\to L\stackrel f\to M\stackrel g\to N\to 0$ be an e-exact sequence of $R$-modules and $R$-homomorphisms. Then $0\to\Gamma _a(L)\stackrel{\Gamma_a(f)}\to \Gamma_a(M)\stackrel{\Gamma_a(g)}\to \Gamma_a(N)$ is an e-exact sequence. Furthermore, if $N$ is a torsion-free module, then the functor $\Gamma_a(\:)$ is an $e$-exact.
\end{lemma}
\begin{proof}
	It is clear that $ker(\Gamma_a(f))=0$. To show that $Im (\Gamma_a(f))\leqslant_e Ker (\Gamma_a(g))$. Let $0\ne y\in Ker(\Gamma_a(g))$. Then $y\in \Gamma_a(M)$ and $g(y)=0$, that is there exist $n\in N$ such that $a^ny=0$. Now, since $Im f\leqslant_e Kerg$, there exists $0\ne r\in R$ such that $f(l)=ry$, for some $l\in L$. Since $f(a^nl)= a^nf(l)= a^n(ry)=0, a^nl=0$ and $l\in \Gamma_a(L)$. Now, to show that $Im(\Gamma_a(g))\leqslant_e \Gamma_a(N)$. Let $y\in \Gamma_a(g)\cap Rx$, for $0\ne x\in \Gamma_a(N)$. Then $y=g(m)=rx$ for $m\in \Gamma_a(M)$ and $r\in R$. By hypothesis, $rx\ne 0$ and so $y\ne 0$ which guarantees the $e$-exactness of the sequence $0\to\Gamma _a(L)\stackrel{\Gamma_a(f)}\to \Gamma_a(M)\stackrel{\Gamma_a(g)}\to \Gamma_a(N)\to 0$.   
	
\end{proof}
\begin{definition}
The right e-derived functors $R^nT$ of a covariant functor $T$ are defined on an $R$-module $A$ by $(R^nT)(A)=H^n(TE^A)=\frac{ker(Td^n)}{Im(Td^{n-1})}$, where $E:0\to A \to E^0\stackrel {d^0} \to E^1 \stackrel {d^1}\to \dots$ is an e-injective resolution of $A$ and $E^A$ is its deleted $e$-injective resolution.
\end{definition}
\quad Define the $n_{th}$ $e$-cohomolgy module $_eH^n_a(M)$ of $M$ with respect to an ideal $\textbf{a}$ as the $n_{th}$ right e-derived functor of the torsion functors $\Gamma_a(\: )$, that is $_eH^n_a(M)=R^n\Gamma_a(M)=\frac{ker(\Gamma_a(d^n))}{Im (\Gamma_a(d^{n-1}))}$. To calculate $_eH^n_a(M)$ with e-injective resolution $0\to M\stackrel{\alpha}\to I^0\stackrel{d^0}\to I^1\to \dots \to I^n\stackrel{d^n}\to I^{n+1}\to \dots$ for any $R$-module $M$ we do the following:  Apply the functor $\Gamma_a$ to the deleted complex of $M$ $I^E:0\stackrel{d^{-1}} \to I^0\stackrel{d^0} \to I^1 \to \dots \to I^n\stackrel{d^n}\to\dots $ and obtain $0\to \Gamma_a(I^0)\stackrel{\Gamma_a (d^0)}\to \Gamma_a(I^1)\to \dots\to\Gamma_a(I^n)\to \dots$ then take the $n_{th}$ $e$-cohomology module of this complex, the result is $_eH^n_a(M)=\frac{Ker(\Gamma_a(d^n))}{Im(\Gamma_a(d^{n-1}))}=(R^nT)(M)=R^n(\Gamma_a(M))=H^n(TE^M)$.
\begin{theorem}
	The right e-derived functors for $\Gamma_a$ are additive functors for every integer $n$.
\end{theorem} 
\begin{proof}
Let $f:M \to M^{'}$ be a morphism. Then by Theorem \ref{bb}, there is a chain map $\breve{f}:E^M\to E^{M^{'}}$ over $f$, where $E^M$ and $E^{M^{'}}$ are deleted e-injective resolutions for $M$ and $M{'}$ respectively. Then  $\Gamma_a\breve{f}:\Gamma_aE^M\to \Gamma_aE^{M^{'}}$ is also a chain map, and so there is a well-defined map $_eH^n_a(f)=(R^n\Gamma_a)f:H^n(\Gamma_aE^M)\to H^n(\Gamma_aE^{M^{'}})$, defined by $_eH^n_a(f)=(R^n\Gamma_a)f=H^n(\Gamma_a\breve{f})=(\Gamma_a\breve{f})$ and $_eH^n_a=R^n\Gamma_a$ is also an additive covariant functor, because $_eH^n_a(f+g)=(R^n\Gamma_a)(f+g)=H^n(\Gamma_a(f+g))=H^n(\Gamma_a(f)+\Gamma_a(g))=H^n(\Gamma_a(f))+H^n(\Gamma_a(g))=  _eH^n_a(f)+  _eH^n_a(g)$. Therefore the right e-derived functors are additive functors for every integer n.
\end{proof}


\begin{proposition}\label{ne}
	Let $A$ be any $R$- module. Then $_eH^n_a(A)=(R^n\Gamma_a)A=0$, for all negative integers $n$.
\end{proposition}
\begin{proof}
	Suppose that $E:0\to A\to E^0 \stackrel{d^0} \to E^1\to \dots$ be an e-injective resolution for $A$. Then the deleted complex of $A$ is $E^A: 0\to E^0\stackrel{d^0}\to E^1\to \dots$. After applying $\Gamma_a$ on the deleted complex, we get $\Gamma_aE^n=0$ for all negative integers n, because the $n_{th}$ term of $E^A$ is zero when $n$ is negative. Hence $_eH^n_a(A)=R^n \Gamma_a(A)=0$ for all negative integers n.
\end{proof}
\begin{corollary}\label{ein}
	Let $E$ be an e-injective $R$-module. Then $_eH^n_a(E)=(R^n \Gamma_a)(E)={0}$, for all $n\geq1$.
\end{corollary}
\begin{proof}
	Since $E$ is an e-injective module, the e-injective resolution of $E$ is $0\to E\stackrel{1_E}\to E\to 0$. The corresponding deleted e-injective resolution $E^E$ is the complex $0\to E\to 0$. Hence, the $n_{th}$ term of $\Gamma_a (E^E)$ is  ${0}$ for all $n\geq1$ and so $_eH^n_a(E)=(R^n \Gamma_a)(E)=H^n(\Gamma_a E^E)={0}$ for all $n\geq1$.
\end{proof}
\begin{theorem}\label{3.9}
	Let $0\to L\stackrel{f} \to M \stackrel {g} \to N \to 0$ be an e-exact sequence of R-modules and R-homomorphisms and $N$ torsion-free. Then, for each $i\in N_0$, there are a connectig homomorphisms ${_eH^i_a(N)}\stackrel{\sigma}\to {_eH^{i+1}_a(L)}$ and these connecting homomorphisms make the resulting long sequence $0\to {_eH^0_a(L)}\stackrel{_eH^0_a(f)}\to {_eH^0_a(M)}\stackrel{_eH^0_a(g)}\to {_eH^0_a(N)}\to {_eH^1_a(L)}\to \dots \to {_eH^i_a(L)}\stackrel{_eH^i_a(f)}\to {_eH^i_a(M)}\to {_eH^i_a(N)}\stackrel{\sigma^*}\to {_eH^{i+1}_a(L)}\to \dots$ $e$-exact.
\end{theorem}
\begin{proof}
	By applying $\Gamma_a$ to an e-exact sequence $0\to L\stackrel{f} \to M \stackrel {g} \to N \to 0$  we obtain an $e$-exact sequence $0\to\Gamma_a(L) \stackrel{\Gamma_a(f)} \to \Gamma_a(M) \stackrel {\Gamma_a(g)} \to \Gamma_a(N)\to 0$ by Lemma \ref{1.13} and by \cite[Theorem 3.1]{ZA} there is a connecting homomorphism $\sigma_n:H^n(\Gamma_a(N))\to H^{n+1}(\Gamma_a(L))$ and by \cite[Theorem 3.2]{ZA} there is a long $e$-exact sequence of $R$-modules and $R$-morphisms $0\to {_eH^0_a(L)}\stackrel{_eH^0_a(f)}\to {_eH^0_a(M)}\stackrel{_eH^0_a(g)}\to {_eH^0_a(N)}\to {_eH^1_a(L)}\to \dots \to {_eH^i_a(L)}\stackrel{_eH^i_a(f)}\to {_eH^i_a(M)}\to {_eH^i_a(N)}\stackrel{\sigma^*}\to {_eH^{i+1}_a(L)}\to \dots$
\end{proof}
\begin{theorem}\label{3.10}
	For any torsion-free $R$-module $M$, $_eH^0_a(M)$ is naturally equivalent to $r \Gamma_a(M)$ for some $r \neq 0 \in R$. 
\end{theorem}
\begin{proof}
	Let $E:0\to A\stackrel{\sigma}\to E^0\stackrel{d^0}\to E^1\stackrel{d^1}\to E^2\to \dots$ be an e-injective resolution for an $R$-module $A$ and $E^A:0\to E^0\stackrel{d^0}\to E^1\stackrel{d^1}\to E^2\to \dots$ be the deleted $e$-injective resolution for $A$. When we apply $\Gamma_a(\:)$ on the deleted $e$-injective resolution we get $0\to \Gamma_a(E^0)\stackrel{d^{0^*}}\to \Gamma_a(E^1)\stackrel{d^{1^*}}\to \Gamma_a(E^2)\to \dots$. Then by the definition of e-cohomology, we have $_eH^0_a(A)=H^0(\Gamma_a(E^A))=Kerd^{0*}$. But the left e-exactness of $\Gamma_a(\:)$ gives an e-exact sequence $0\to \Gamma_a(E^0)\stackrel{d^{0^*}}\to \Gamma_a(E^1)\stackrel{d^{1^*}}\to \Gamma_a(E^2)\to \dots$. We define $\sigma^*:\Gamma_a(A)\to Kerd^{0*}$. Since $Im \sigma^*\leqslant_e Kerd^{0*}$, $\sigma^*$ is well-defined and $\Gamma_a(\:)$ a left e-exact functor, $\sigma^*$ is monic. Now, we want to prove that $\sigma^*$ is epic. Let $x\in Kerd^{0*}$. Then $d^{0*}(x)=d^0(x)=0$. Therefore $x\in Kerd^0$. By e-exactness of the e-injective resolution we have $Im\sigma \leqslant_e Kerd^0$, so there exist $a'\in A$ and $0\ne r\in R$ such that $\sigma(a')=rx\ne 0$. Now, we define $f:r\Gamma_a(A)\to A$ by $f(ry)=a'$. Let $y_1,y_2 \in \Gamma_a(A)$ with $ry_1=ry_2$. Then $\sigma(a_1')=\sigma(a_2')$. By monicness of $\sigma$ we have $a_1'=a_2'$ and so $f$ is well-defined. Now we have $rx=\sigma(a')=\sigma(rf(y))=r\sigma(f(y))=r\sigma^*(f(y))$ which is equivalent to $\sigma^*(f(y))=x$. Hence $\sigma^*$ is an isomorphism and since $_eH^0_a(A)=Kerd^{0*}$. Therefore $_eH^0_a(\:)$ is isomorphic to $r \Gamma_a(\:)$ for some nonzero $r$ in $R$.
\end{proof}
\begin{corollary}
   If $0\to L\stackrel{f} \to M \stackrel {g} \to N \to 0$ is an e-exact sequence of R-modules where $N$ is torsion-free, then there is a long e-exact sequence $0\to {_eH^0_a(L)} \stackrel{\Gamma_a(f)}\to {_eH^0_a(M)}\stackrel{\Gamma_a(g)} \to {_eH^0_a(N)}\stackrel{\sigma}\to {_eH^1_a(L)}\stackrel{_eH^1_a(f)}\to {_eH^1_a(M)}\stackrel{_eH^1_a(g)}\to \dots$. In addition, if $L$, $M$ and $N$ are torsion-free modules, then there is a long e-exact sequence $0\to r_1 \Gamma_a(L) \stackrel{\Gamma_a(f)}\to r_2 \Gamma_a(M)\stackrel{\Gamma_a(g)} \to r_3 \Gamma_a(N)\stackrel{\sigma}\to {_eH^1_a(L)}\stackrel{_eH^1_a(f)}\to {_eH^1_a(M)}\stackrel{_eH^1_a(g)}\to \dots$ for some nonzero $r_1, r_2, r_3 $ in $R$.
\end{corollary}
\begin{proof}
	Directly follows from Theorem \ref{3.9} and Theorem \ref{3.10}. 
\end{proof}
\begin{theorem}
	Given an e-commutative diagram of R-modules having e-exact rows where $A''$ and $C''$ are torsion-free:
	\[ \begin{tikzcd}[arrows={-Stealth}]
		0\rar & A'\rar["i"]\dar["f"] & A\rar["p"]\dar["g"] & A''\dar["h"]\rar & 0\\%
		0\rar& C'\rar & C\rar &	C''\rar &0
	\end{tikzcd}	\] Then there is an e-commutative diagram with e-exact rows, 
	\[ \begin{tikzcd}[arrows={-Stealth}]
		\dots \to {_eH^n_a(A')}\rar["i_*"]\dar["f_*"] & {_eH^n_a(A)}\rar["p_*"]\dar["g_*"] & {_eH^n_a(A'')}\rar ["\sigma"]\dar["h_*"] & {_eH^{n+1}_a(A')}\rar \dar["f_*"]& \dots \\%
		\dots \to {_eH^n_a(C')}\rar["j_*"] & {_eH^n_a(C)}\rar["q_*"] &	{_eH^n_a(C'')}\rar ["\sigma^*"]\rar & {_eH^{n+1}_a(C')}\rar&\dots 
	\end{tikzcd}	\]
\end{theorem}
\begin{proof}
	By Proposition \ref{1.8}, we have an e-exact sequence of deleted complexes $0\to E^{A'} \to E^A \to E^{A''}\to 0$. If $T=\Gamma_a$, then $0\to TE^{A'} \to TE^A \to TE^{A''}\to 0$ is still an e-exact by Lemma \ref{1.13}. By \cite[Remark 3.3]{ZA} there is an $e$-commutative diagram of $R$-modules and $R$-morphisms
		\[ \begin{tikzcd}[arrows={-Stealth}]
		\dots \to H^{n-1}(\Gamma_a(A'))\rar["i_*"]\dar["f_*"] & H^{n-1}(\Gamma_a(A))\rar["p_*"]\dar["g_*"] & H^{n-1}(\Gamma_a(A''))\rar ["\sigma"]\dar["h_*"] & H^n(\Gamma_a(A'))\rar \dar["f_*"]& \dots \\%
		\dots \to H^{n-1}(\Gamma_a(C'))\rar["j_*"] & H^{n-1}(\Gamma_a(C))\rar["q_*"] &	H^{n-1}(\Gamma_a(C''))\rar ["\sigma^*"]\rar & H^n(\Gamma_a(C'))\rar&\dots 
	\end{tikzcd}	\] and the rest will achieved  by applying the definition of $e$-cohomology $_eH^n_a(\:)=H^n(\Gamma_a(E^A))$.
\end{proof}
\begin{corollary}
	Let $M$ be an $a$-torsion $R$-module. Then there exists an $e$-injective resolution of $M$ in which each term is an $a$-torsion $R$-module.
\end{corollary}
In local cohomology, every $a$-torsion $R$-module $M$ has zero $i_{th}$ local cohomology modules, that is  $H^i_a(M)=0$, for all $i>0$. while in $e$-cohomology this is not true in general. To be more precise we present the following example.
\begin{example}\label{ex}
	Consider the $e$-injective resolution $0\to \frac{\mathbb{Z}}{2\mathbb{Z}}\stackrel{f} \to \frac{\mathbb{Z}}{8\mathbb{Z}} \stackrel{g}\to \frac{\mathbb{Z}}{16\mathbb{Z}}\to {0}$ for the $\mathbb{Z}$-module $\frac{\mathbb{Z}}{2\mathbb{Z}}$ where $f(1+2\mathbb{Z})=4+8\mathbb{Z}$ and $g(n+8\mathbb{Z})=8n+16\mathbb{Z}$. Since each term of the resolution is $2\mathbb{Z}$-torsion while the $e$-cohomology module $H^1_{2\mathbb{Z}}(\frac{\mathbb{Z}}{2\mathbb{Z}})=\frac{\mathbb{Z}}{8\mathbb{Z}}$ is non-zero. 
\end{example}
At the end of the paper and as a future work one can do research on the $e$-cohomology dimension and its connection with the cohomolgy dimension.\\

\textbf{Acknowlegments}: We would like to thank the referees for their thoughtful comments and efforts towards
improving the manuscript.
\vskip 0.4 true cm


\begin{thebibliography}{17}
\bibitem{AZ} I. Akray and A. Zebari, Essential exact sequence, Commun. Korean Math. Soc., 35 (2) (2020), 469-480.
\bibitem{ZA} I. Akray and A. Zebari, The Homology regarding to $e$-exact sequences, Commun. Korean Math. Soc., 38 (2023). 
\bibitem{MR} M. P. Brodmann and R. Y. Sharp, Local cohomology an algebraic introduction with geometric applications, Second edition, Cambridge studies in advanced mathematics, 2013.
\bibitem{CF} F. Campanini and A. Facchini, Exactness of cochain complexes via additive functors, Commun. Korean Math. Soc., 35 (4)  (2020), 1075-1085.
\bibitem{DP} B. Davvaz and Y. A. Parnian-Garamaleky, A note on exact sequences, Bull. Malaysian Math. Soc., 22 (1999), 53-56.
\bibitem{G} A. Grothendieck, Local cohomology, Lecture notes in mathematics 41 Springer, Berlin, 1967.
\bibitem{H} K. S. Harinath, Semi-split sequences-I, Proceedings of the indian academy of sciences- section A, 75 (1972), 13-18.
\bibitem{R} J. J. Rotman, An introduction to homological algebra, Secon edition, University of Illinois-champaign, Springer Science, New York, 2009.


\end{thebibliography}
\end{document}